\documentclass{mystyle}
\newcommand{\area}{\operatorname{area}}
\newcommand{\cyc}{\operatorname{cyc}}
\newcommand{\perim}{\operatorname{perim}}
\newcommand{\holes}{\operatorname{holes}}
\newcommand{\innerperim}{\operatorname{innerperim}}
\newcommand{\outerperim}{\operatorname{outerperim}}
\newcommand{\comps}{\operatorname{comps}}

\newcommand{\defterm}[1]{\blue{\emph{#1}}}

\title[Maximum number of cycles in triangular billiards]{The maximum number of cycles in a triangular-grid billiards system with a given perimeter}

\author[H. Zhu]{Honglin Zhu}

\begin{document}

\begin{abstract}
    Given a grid polygon $P$ in a grid of equilateral triangles, Defant and Jiradilok considered a billiards system where beams of light bounce around inside $P$. We study the relationship between the perimeter $\operatorname{perim}(P)$ of $P$ and the number of different trajectories $\operatorname{cyc}(P)$ that the billiards system has. Resolving a conjecture of Defant and Jiradilok, we prove the sharp inequality $\operatorname{cyc}(P) \leq (\operatorname{perim}(P) + 2)/4$ and characterize the equality cases. 
\end{abstract}

\maketitle

\section{Introduction}\label{sec_intro}
Equip the plane with a lattice of equilateral triangles of side length one, oriented so that some of the grid lines are horizontal. Following Defant and Jiradilok \cite{defant2023}, we refer to the sides of the grid cells as \defterm{panes}. Define a \defterm{grid polygon} to be a simple polygon (i.e., homeomorphic to a disk) whose boundary is a union of panes. 

Let $P$ be a grid polygon (not necessarily convex) with boundary panes $b_1, \ldots, b_n$ listed clockwise. For any boundary pane $b_i$, emit a laser beam from the midpoint of $b_i$ into the interior of $P$ so that it forms a $60^{\circ}$ angle with $b_i$. The direction depends on the orientation of $b_i$ (viewed as a unit vector going from the endpoint adjacent to $b_{i- 1}$ to the endpoint adjacent to $b_{i+1}$) and makes either a $60^{\circ}$ angle, a $180^{\circ}$ angle, or a $300^{\circ}$ angle with the positive $x$-axis (see \Cref{fig_first_eg}). The light beam travels through the interior of $P$ until reaching the midpoint of another boundary pane $b_{j}$, meeting it at a $60^{\circ}$ angle. This defines a permutation $\pi: [n] \to [n]$ where the light from the midpoint of $b_i$ hits the midpoint of $b_{\pi(i)}$. Defant and Jiradilok call $\pi$ the \defterm{billiards permutation} of $P$. 

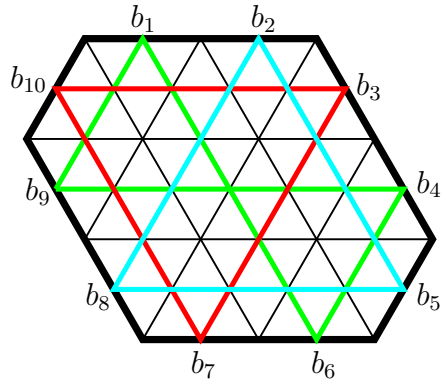
\begin{figure}[ht]
    \centering
    \begin{tikzpicture}
	\node (image) at (0,0) {
		\includegraphics[width=0.4\textwidth]{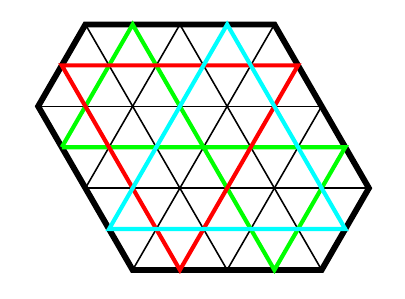}
	};
    \node[text width=3cm] at (0.2,2.3) 
    {$b_1$};
    \node[text width=3cm] at (1.8,2.3) 
    {$b_2$};
    \node[text width=3cm] at (3.2,1.4) 
    {$b_3$};
    \node[text width=3cm] at (4.0,0.1) 
    {$b_4$};
    \node[text width=3cm] at (4.0,-1.4) 
    {$b_5$};
    \node[text width=3cm] at (2.6,-2.3) 
    {$b_6$};
    \node[text width=3cm] at (1.0,-2.3) 
    {$b_7$};
    \node[text width=3cm] at (-0.4,-1.4) 
    {$b_8$};
    \node[text width=3cm] at (-1.2,0.0) 
    {$b_9$};
    \node[text width=3cm] at (-1.4,1.5) 
    {$b_{10}$};
\end{tikzpicture}
\caption{A grid polygon with $10$ boundary panes. The billiards permutation of this polygon is $(1 \; 6 \; 4 \; 9) \; (2 \; 5 \; 8) \; (3 \; 10 \; 7)$.}
\label{fig_first_eg}
\end{figure}

The billiards permutation can be interpreted as a ``billiards process'' as follows. Imagine that the boundary panes of $P$ are mirrors and that all other panes in the lattice are transparent. When the light beam from $b_i$ reaches $b_{\pi(i)}$, it reflects off the mirror so that the beam forms a $60^{\circ}$ angle with $b_{\pi(i)}$. This reflected beam will subsequently travel to $b_{\pi^2(i)}$ and so on until eventually coming back to $b_i$. Then the cycles in $\pi$ correspond precisely to the different trajectories of beams of light. In particular, observe that each $i \in [n]$ appears in exactly one cycle, and each cycle has length at least $3$. 

Much of previous work on billiards in planar regions comes from the fields of dynamical systems and recreational mathematics, where beams of light could have arbitrary initial positions and directions \cite{boldrighini1978,croft1963,croft1991,detemple1981,detemple1983,farber2002,gutkin1986,halpern1977,sine1979}. In contrast, Defant and Jiradilok's \cite{defant2023} framework imposes much greater rigidity. While doing so renders many dynamical system problems trivial or uninteresting, it also gives rise to many new and fascinating combinatorial and geometric questions. 

More broadly, there has been increasing interest to study combinatorial objects through a dynamical system setup \cite{adams2024,barkley2024,barkley2021,defant2023}. These recent works form a nascent subfield within dynamical algebraic combinatorics, dubbed ``combinatorial billiards'' by \cite{adams2024}.

In this paper, we are interested in the relationship between the following quantities associated to a grid polygon $P$: 
\begin{itemize}
    \item the \defterm{perimeter} of $P$, denoted $\perim(P)$, which is the number of boundary panes of $P$, and
    \item the number of cycles in the associated permutation $\pi_{P}$, denoted $\cyc(P)$, which is also the number of different light beam trajectories in the associated billiards system of $P$. 
\end{itemize}

As in \cite{defant2023}, we say that a grid polygon is \defterm{primitive} if there is no interior pane that cuts the polygon into two disconnected regions. Defant and Jiradilok proved the following upper bound on $\cyc(P)$ in terms of $\perim(P)$. 
\begin{theorem}[\cite{defant2023}, Theorem~1.2]\label{thm_colin}
    If $P$ is a simple grid polygon, then
    \[
        \cyc(P) \leq \frac{2}{7} \left(\perim(P) + \frac{3}{2} \right).
    \]
\end{theorem}
They noted that this inequality is not tight, and conjectured that $\cyc(P) \leq \frac{1}{4}(\perim(P) + 2)$. We prove this conjecture and give a partial characterization of the equality cases. 
\begin{theorem}\label{thm_perim}
    If $P$ is a simple grid polygon, then
    \[
        \cyc(P) \leq \frac{\perim(P) + 2}{4}.
    \]
    Furthermore, if $P$ is primitive, then 
    \[
        \cyc(P) = \frac{\perim(P) + 2}{4}
    \]
    if and only if the billiards permutation of $P$ contains exactly two $3$-cycles and no cycles of length greater than $4$. 
\end{theorem}
Defant and Jiradilok \cite{defant2023} observed that an infinite family of equality cases exists, given by trees of unit hexagons. Furthermore, Defant and Jiradilok\footnote{Private communications.} also found an infinite family of simple primitive grid polygon whose billiards permutation contains exactly two $3$-cycles and $k$ $4$-cycles, for each $k \geq 0$ (see \Cref{fig_equality_case}). In fact, for each $p \geq 3$, we can easily extend this construction to obtain a simple primitive grid polygon $P$ with $\perim(P) = p$ and 
\[
    \cyc(P) = \left\lfloor\frac{p + 2}{4}\right\rfloor.
\]
In particular, for $p = 4k$, simply take a rhombus of side length $k$. The resulting grid polygon has $k$ $4$-cycles. For $p = 4k + 1$, take a rhombus of side length $k + 1$ and cut off a triangle of side length $1$ on the top and a triangle of side length $2$ on the bottom. The resulting grid polygon has one $3$-cycle, one $6$-cycle, and $k - 2$ $4$-cycles (for $k = 1$ there is just a single $5$-cycle). For $p = 4k + 3$, take a rhombus of side length $k + 1$ and cut off the top unit triangle. Then the resulting grid polygon has one $3$-cycle and $k$ $4$-cycles. 

\begin{figure}
    \centering
    \includegraphics[width = 0.2\textwidth]{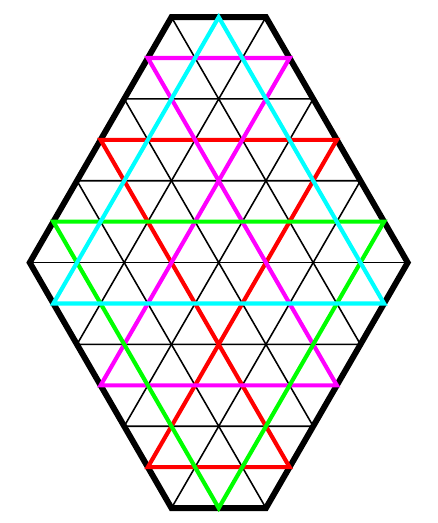}
    \caption{Take a rhombus of side length $k + 2$ and cut off the two antipodal unit triangles. The resulting grid polygon contains exactly two $3$-cycles and $k$ $4$-cycles. The figure is the case $k = 2$.}
    \label{fig_equality_case}
\end{figure}

Defant and Jiradilok's setup has an equivalent formulation in terms of Postnikov's plabic graphs introduced in \cite{postnikov2006}. Defant and Jiradilok's original motivation for considering the triangular-grid billiards problem was to better understand a class of plabic graphs. For details on how plabic graphs are related to our problem, see \cite{defant2023}. For more background on plabic graphs, see \cite{postnikov2006}. 

Lam and Postnikov \cite{lam2020} recently introduced combinatorial objects called \defterm{membranes}, which are dual to plabic graphs. In particular, minimal membranes of type $A_2$ (we recall their definition in \Cref{sec_prelim}) can be viewed as a generalization of the class of grid polygons. We will prove \Cref{thm_perim} by proving an analogous inequality for a class of minimal membranes of type $A_2$ we call \defterm{generalized grid polygons}. 

The paper is organized as follows. In \Cref{sec_prelim}, we introduce generalized grid polygons and build the necessary infrastructure around them. The main reason we consider them is that we would like to ``drop'' a cycle from the billiards permutation of a grid polygon (see \Cref{thm_dropping_cycle}), and doing so may result in a polygon that is no longer simple. In \Cref{sec_perim}, we prove the generalization of \Cref{thm_perim} for generalized grid polygons. In \Cref{sec_future}, we discuss future directions and open problems. 

\subsection*{Acknowledgements}
This work was done at the University of Minnesota Duluth with support from Jane Street Capital, the National Security Agency, and the CYAN Undergraduate Mathematics Fund at MIT. The author thanks Joe Gallian and Colin Defant for organizing the Duluth REU and providing this great research opportunity. We thank Mitchell Lee for detailed discussions and comments that significantly improved this paper. We thank Evan Chen for making the observation that led to \Cref{thm_dropping_cycle}, Maya Sankar for suggesting the use of simplicial complexes, and Colin Defant and Pakawut Jiradilok for providing comments.  

\section{Preliminaries}\label{sec_prelim}
To prove \Cref{thm_perim}, we will deal with polygons that are not simple. Thus, we need to formalize and extend the definition of grid polygons. In particular, we would like to include polygons that can overlap with themselves in certain ways. Intuitively, we can view our generalized polygon as an object obtained by gluing unit triangles together in a well-behaved way. To make this precise, we use the concept of \defterm{simplicial complexes}. 
\begin{definition}[\cite{spanier1966}]\label{def_simp_complex}
    A \defterm{simplicial complex} $X$ consists of a set of \defterm{vertices} and a set of finite nonempty subsets of vertices called \defterm{simplices} such that:
    \begin{itemize}
        \item any set consisting of exactly one vertex is a simplex, and
        \item any nonempty subset of a simplex is a simplex.
    \end{itemize}

    A simplex containing exactly $d + 1$ vertices is a $d$-simplex, and we say the \defterm{dimension} of such a simplex is $d$. The dimension of a simplicial complex $X$ is the largest dimension of any simplex of $X$. A $d$-dimensional simplicial complex $X$ is \defterm{homogeneous} if every simplex in $X$ is a face (a subset) of some $d$-dimensional simplex in $X$. 

    For two simplicial complexes $X_1$ and $X_2$, a \defterm{simplicial map} $f: X_1 \to X_2$ is a function from the vertices of $X_1$ to the vertices of $X_2$ such that the image of any simplex is a simplex. 
\end{definition}
For more on simplicial complexes, refer to textbooks on algebraic topology (e.g. \cite{spanier1966}). Observe that there is a natural way to view the triangular grid on the plane as a two-dimensional homogeneous simplicial complex, where the vertices are the grid points, the $1$-simplices are the panes, and the $2$-simplices are the unit grid triangles. Call this simplicial complex $T$.
\begin{definition}\label{def_polygon}
    A \defterm{generalized grid polygon} is a finite simply-connected homogeneous simplicial complex $X$ of dimension $2$ together with a simplicial map to the triangular grid $f: X \to T$ such that: 
    \begin{enumerate}
        \item the geometric realization of $X$ is a wedge of disks along boundary points;
        \item each simplex in $X$ is mapped by $f$ to a simplex of the same dimension in $T$; 
        \item each interior vertex is contained in exactly six $2$-simplices; and
        \item each $1$-simplex that is not on the boundary of $X$ is contained in exactly two $2$-simplices, whose images under $f$ form a diamond.
    \end{enumerate}
\end{definition}

\begin{figure}
    \centering
    \includegraphics[width = 0.4\textwidth]{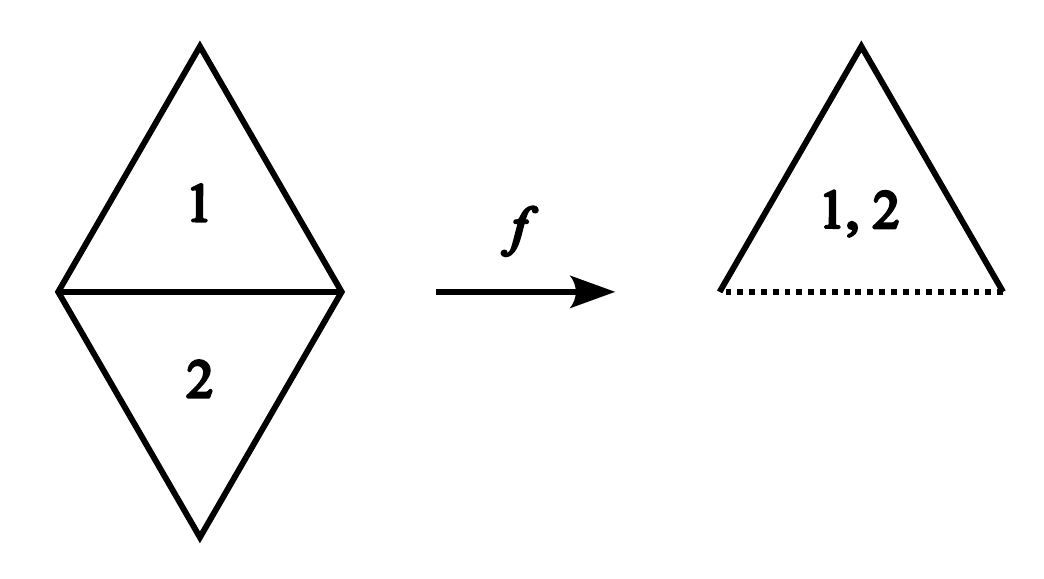}
    \caption{A simplicial complex whose simplicial map violates \Cref{def_polygon} (4). Intuitively, we do not want “folding” to happen.}
    \label{fig_bad_eg}
\end{figure}

\begin{figure}
    \centering
    \includegraphics[width = 0.2\textwidth]{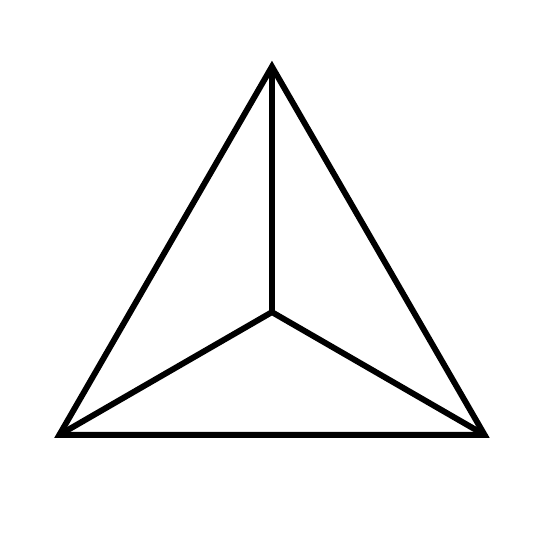}
    \caption{A simplicial complex which violates \Cref{def_polygon} (3) and has no simplicial map to the triangular grid.}
    \label{fig_bad_eg_2}
\end{figure}

\begin{figure}[ht]
    \centering
    \includegraphics[width = 0.7 \textwidth]{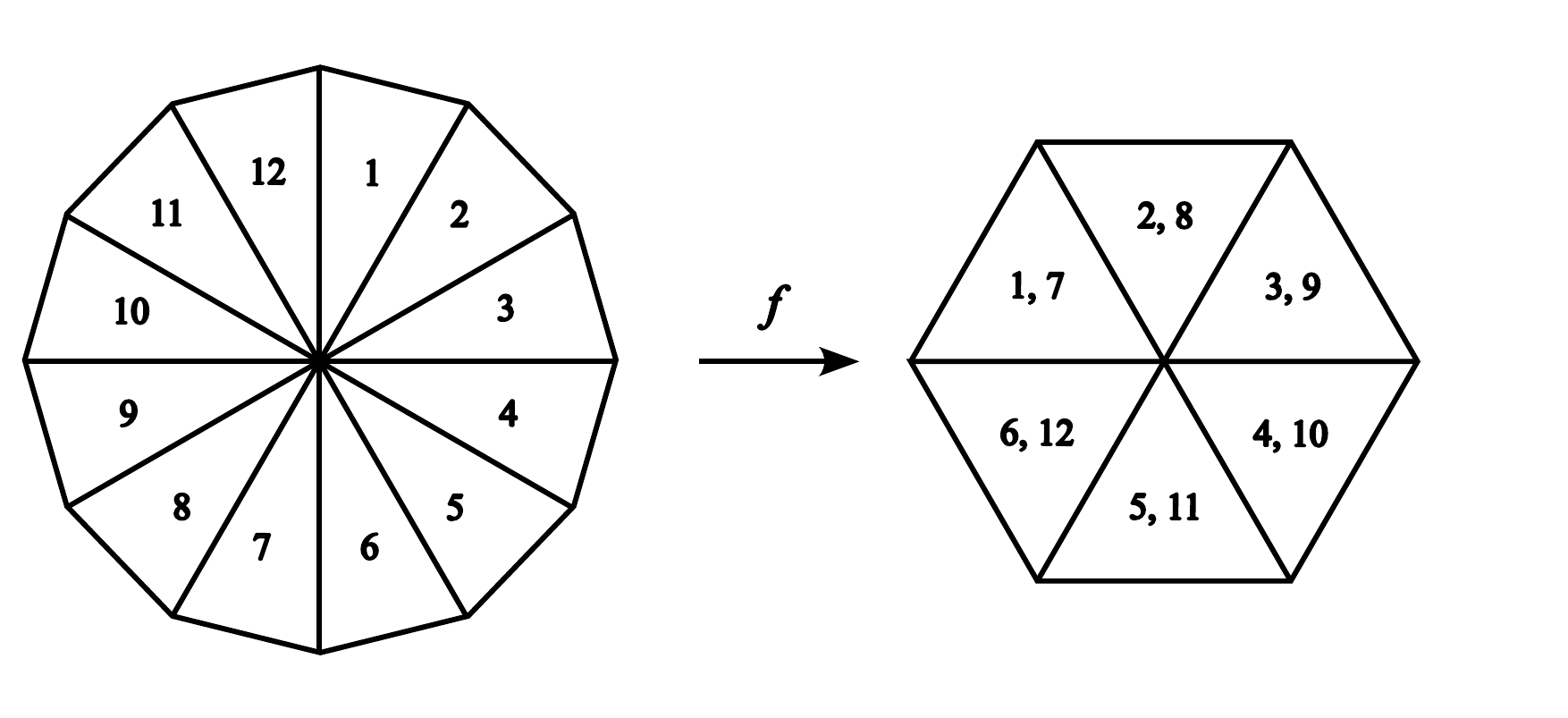}
    \caption{A simplicial complex which has a simplicial map to the triangular grid but violates \Cref{def_polygon} (3).}
    \label{fig_good_eg_2}
\end{figure}

\begin{figure}[ht]
    \centering
    \includegraphics[width = 0.7 \textwidth]{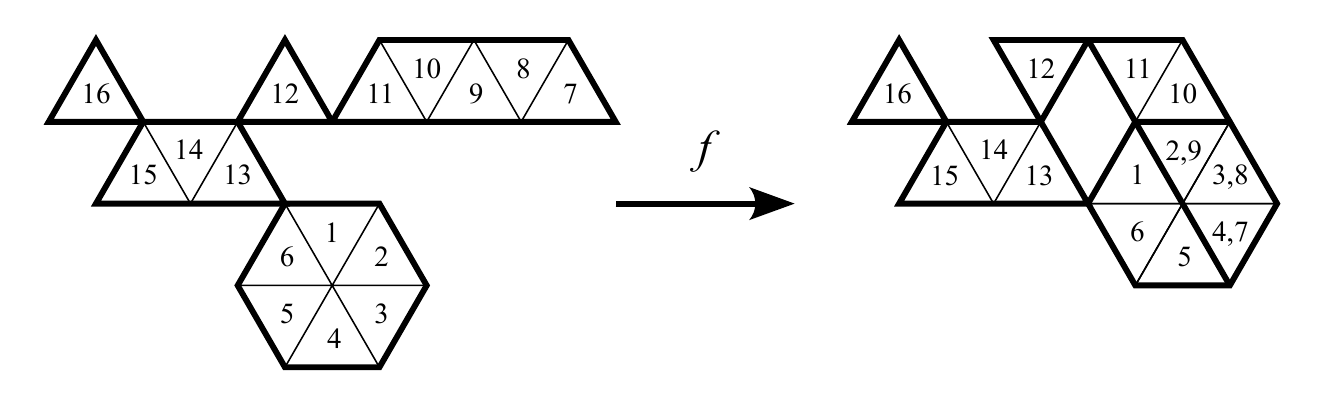}
    \caption{A valid generalized polygon. There are $5$ components. Observe that $f$ maps the simplicial complex to the triangular grid in a way that overlaps some $2$-simplices.}
    \label{fig_good_eg}
\end{figure}

In particular, observe that simple grid polygons can be viewed as generalized grid polygons. See \Cref{fig_bad_eg,fig_bad_eg_2,fig_good_eg_2,fig_good_eg} for some non-examples and examples of generalized grid polygons. For two generalized grid polygons $X$ and $Y$ with simplicial maps $f$ and $g$ and boundary vertices $x \in X$ and $y \in Y$, the \defterm{wedge} $X \wedge Y$ of $X$ and $Y$ is a generalized grid polygon which is the wedge of the simplicial complexes $X$ and $Y$ at basepoints $x$ and $y$, with simplicial map $h: X \wedge Y \to T$ such that for any $z \in X$, $h(z) = f(z)$ and for any $w \in Y$, $h(w) = g(w) - g(y) + f(x)$. 

We call a generalized grid polygon \defterm{indecomposable} if it is not the wedge of any two generalized grid polygons and \defterm{decomposable} otherwise. Note that a generalized grid polygon is indecomposable if and only if it is homeomorphic to a disk. A decomposable generalized grid polygon is the wedge of some indecomposable generalized grid polygons along boundary vertices, which we call its \defterm{components}. Furthermore, a component is an \defterm{outermost component} if it is wedged to the rest of the generalized grid polygon at a single point. Consider the graph whose vertices are the components, with an edge between any two components that are wedged together. This graph is a tree by the simply-connectedness of a generalized grid polygon, and its leaves are the outermost components. We denote by $\comps(P)$ the number of components of a generalized grid polygon $P$. 

For a $1$-simplex on the boundary of a generalized grid polygon, we view its image (a pane in the triangular grid) as a vector placed at the pane and oriented so that its right side is the interior of the polygon. We say that $(b_1, \ldots, b_n)$ is a \defterm{loop of panes} if each $b_i$ is a pane in the triangular grid and the head of $b_i$ is at the tail of $b_{i + 1}$ (indices taken modulo $n$). Note that the image of the boundary of a generalized grid polygon under the map $f$ given in \Cref{def_polygon} is a loop of panes which can be taken to wind clockwise around the image of each component.

We define the \defterm{billiards permutation} on a generalized grid polygon $P$ in a way that extends Defant and Jiradilok's \cite{defant2023} billiards permutation for simple grid polygons. For any $1$-simplex on the boundary of a generalized grid polygon, emit a laser beam from the midpoint of the $1$-simplex into the interior of $P$ so that, when viewed under the map $f$, the laser beam makes a $60^{\circ}$ angle with the pane and continues in the direction as specified in \Cref{sec_intro}. It keeps going in the interior of (the geometric realization of) $P$ in such a way that its image under $f$ is a straight segment, until it hits another boundary $1$-simplex. It can be checked that this defines a permutation which agrees with the billiards permutation of Defant and Jiradilok in the case of simple grid polygons. Observe that there are no $2$-cycles in a billiards permutation. 

To define the billiards system on generalized grid polygons, we first give a \defterm{labeling} of the $1$-simplices of a generalized grid polygon as follows. Label a $1$-simplex whose image under $f$ forms a $0^{\circ}$ ($60^{\circ}$, $120^{\circ}$, respectively) angle with the positive $x$-axis by $1$ ($2$, $3$, respectively). Notice that each $2$-simplex has its edges labeled $1, 2, 3$ in clockwise order. 

The billiards permutation can be described fully at the level of the simplicial complex, as long as a labeling is given. For a light beam coming into a $2$-simplex labeled $i$, it hits the $1$-simplex labeled $i - 1$ in the same $2$-simplex if it is mapped under $f$ to an upward-pointing triangle and the $1$-simplex labeled $i + 1$ in the same $2$-simplex if it is mapped under $f$ to a downward-pointing triangle (taken modulo $3$). If the $1$-simplex it hits is not on the boundary, the light beam travels into the adjacent $2$-simplex. Otherwise, it bounces off and heads into the same $2$-simplex. 

An important fact is that the boundary alone is not enough to determine the cycle structure of a generalized grid polygon: there exist generalized grid polygons with the same boundary but different (valid) billiards permutations (see \Cref{fig_two_possibilities}). 

\begin{figure}[ht]
    \centering
    \includegraphics[width=0.8 \textwidth]{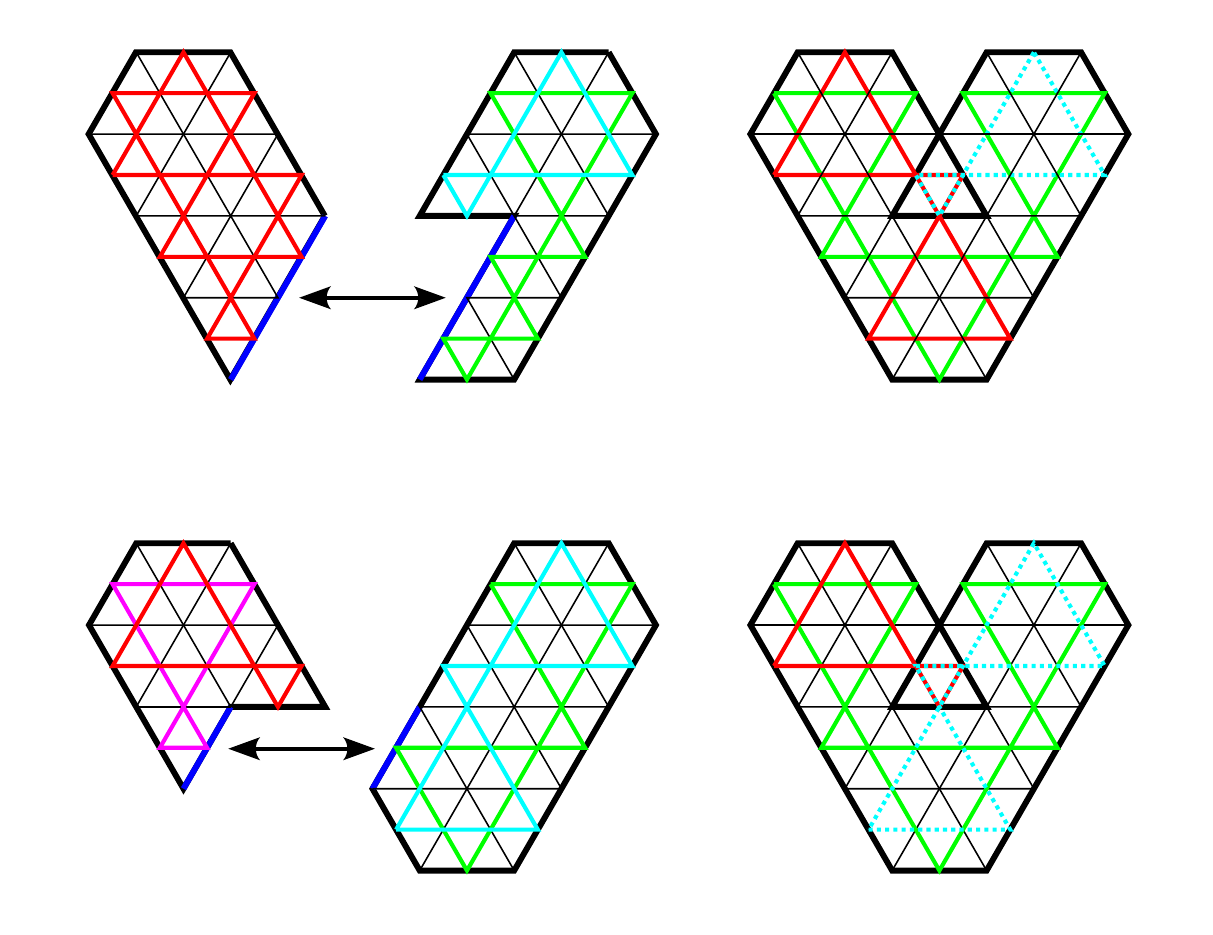}
    \caption{Two grid polygons with the same boundary but different billiards permutations. They can each be seen as gluing two simple polygons together along the blue segments. This means that they are different as simplicial complexes.}
    \label{fig_two_possibilities}
\end{figure}

A \defterm{horizontal strip} is a generalized grid polygon whose image is a trapezoid of ``height'' one whose parallel sides are horizontal. By \Cref{def_polygon} (3), any $2$-simplex $\Delta$ of $P$ is adjacent to at most one $2$-simplex on its left and at most one $2$-simplex on its right (sharing its edges labeled $2$ or $3$). Thus, starting with $\Delta$, we can find all the $2$-simplices connected to it in this way, and call the subcomplex spanned by them $h_{\Delta}$. Observe that $h_{\Delta}$ is a horizontal strip. We call it a \defterm{horizontal strip of $P$}. Then each $2$-simplex of $P$ belongs to exactly one horizontal strip of $P$. Furthermore, a horizontal strip of $P$ is precisely the set of $2$-simplices crossed by a west-going beam of $\pi$. We show that there is a tree structure on the set of horizontal strips of $P$.
\begin{lemma}\label{lem_tree_of_strips}
    Let $P$ be an indecomposable generalized grid polygon and let $S$ be the set of horizontal strips of $P$. Let $\Gamma_{S}$ be the following edge-labeled graph on $S$. For each pair of distinct horizontal strips $h_1, h_2 \in S$, add an edge between $h_1$ and $h_2$ labeled by the multiset $\{p_1, p_2\}$ for each maximal horizontal path $p_1 = p_2 \subset h_1 \cap h_2$. Then $\Gamma_{S}$ is a tree. In particular, there is at most one edge between any two vertices.
\end{lemma}
The reason for this seemingly redundant notation is that we would also like to reconstruct a generalized grid polygon from the information contained in a graph of the form $\Gamma_{S}$ (the horizontal strips as the vertices and the gluing information as the edge labels), whose horizontal strips are not known a priori to share $1$-simplices (see \Cref{lem_tree_reconstruction}). 
\begin{proof}
    Consider any edge in $\Gamma_{S}$ labeled $\{p_1, p_2\}$ so that $p_1 = p_2 \subset h_1 \cap h_2$. We claim that the endpoints of $p_1$ are on the boundary of $P$ and the other vertices of $p_1$ are in the interior of $P$. For a vertex $v$ in the interior of $p_1$, the two horizontal $1$-simplices incident to it are in the interior of $P$ because they are shared by $h_1$ and $h_2$, and the other $1$-simplices incident to $v$ are either in the interior of $h_1$ or that of $h_2$.

    \begin{figure}[ht]
        \centering
        \includegraphics[width = 0.4 \textwidth]{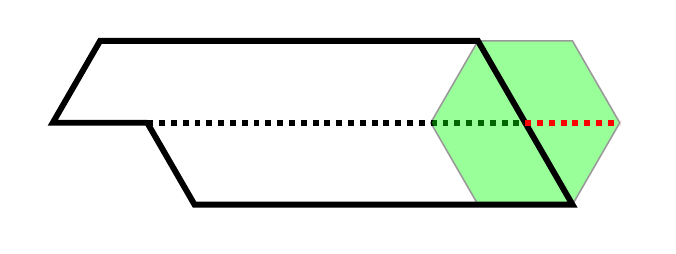}
        \caption{If $v$ is an interior vertex, \Cref{def_polygon} (3) implies that it is contained in a hexagon, which means $v$ cannot be an endpoint in the path.}
        \label{fig_jordan_curve}
    \end{figure}
    
    Now suppose $v$ is an endpoint of $p_1$. For the sake of contradiction, assume that it is an interior vertex. Then by (3) of \Cref{def_polygon} it must be contained in exactly six $2$-simplices whose images under the simplicial map form a unit hexagon. However, this means that there are exactly two horizontal strips that contain $v$ (each containing three of the $2$-simplices containing $v$). These horizontal strips share both horizontal $1$-simplices incident to $v$, contradicting the assumption that $v$ is an endpoint (see \Cref{fig_jordan_curve}). Hence, the horizontal path joins two boundary vertices of $P$, with the interior of the path in the interior of $P$. Since $P$ is indecomposable, it is homeomorphic to a disk. Then by the Jordan curve theorem, if we remove this horizontal path, we disconnect $P$. It follows that $\Gamma_{S}$ must be a tree. 
\end{proof}

We can generalize this tree structure on the horizontal strips to a decomposable generalized grid polygon $P$. For each component, construct the tree according to \Cref{lem_tree_of_strips}. Now for any pair of components wedged together at point $x$, choose a horizontal strip containing $x$ in both components (this may not be unique). Then add an edge between these two horizontal strips labeled $\{x, x, \alpha\}$, where $\alpha \in [4]$ specifies one of the four possible orientations the strips are attached together. Call this edge-labeled graph a \defterm{tree of horizontal strips} of $P$. Conversely, if we are given a tree of horizontal strips, we can reconstruct a generalized grid polygon. 

\begin{lemma}\label{lem_tree_reconstruction}
    Let $S$ be a set of horizontal strips. Let $\Gamma_{S}$ be an undirected, edge-labeled graph on vertex set $S$, where an edge $h_1 h_2$ is labeled by $\{p_1 \subset h_1, p_2 \subset h_2\}$ such that $p_1, p_2$ are horizontal paths of $1$-simplices of the same length, or by $\{x, y, \alpha\}$ such that $x$ is a vertex of $h_1$, $y$ is a vertex of $h_2$, and $\alpha \in [4]$ specifies one of four possible orientations. Suppose the following hold:
    \begin{enumerate}
        \item each $1$-simplex of each horizontal strip of $S$ appears at most once in the horizontal paths in the edge labels; and
        \item the graph $\Gamma_{S}$ is a tree.
    \end{enumerate}
    Then there exists a generalized grid polygon $P$ such that $\Gamma_{S}$ is a tree of horizontal strips of $P$.
\end{lemma}
\begin{proof}
    We can form a simplicial complex on the union of the elements of $S$ by identifying $p_1$ with $p_2$ for each edge labeled $\{p_1, p_2\}$. We arbitrarily fix the image of some horizontal strip under the simplicial map. Using \Cref{def_polygon} (4), we can recover $f$ inductively. A vertex is an interior vertex if and only if it is in the interior of some path in an edge label. This means that it is contained in exactly six $2$-simplices, three from each horizontal strip it is contained in. This shows that \Cref{def_polygon} (3) holds. \Cref{def_polygon} (2) holds trivially and (1) holds inductively (each time we add a new horizontal strip to our tree, the simplicial complex remains homeomorphic to a wedge of disks).
\end{proof}

\begin{figure}[ht]
    \centering
    \includegraphics[width = 0.8 \textwidth]{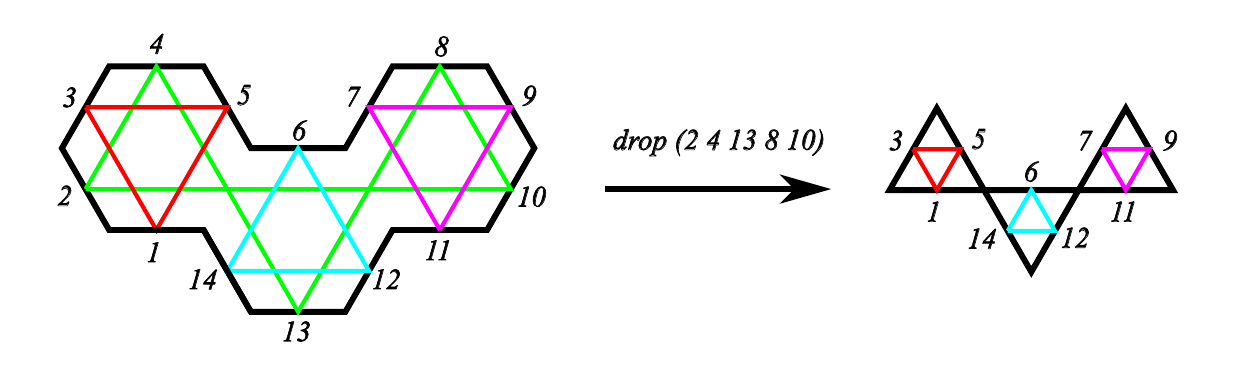}
    \caption{In this diagram, if we drop the green cycle, we end up with another generalized grid polygon. Note that the remaining cycles and boundary panes are preserved.}
    \label{fig_forest_poly}
\end{figure} 

This section culminates with a crucial result about generalized grid polygons that helps us prove \Cref{thm_perim}: we can ``drop'' a cycle from the billiards permutation of a generalized grid polygon and end up with another generalized grid polygon (see \Cref{fig_forest_poly}). 

\begin{theorem}\label{thm_dropping_cycle}
    Let $P$ be a generalized grid polygon with boundary loop $(b_1, \ldots, b_n)$ and billiards permutation $\pi$. Let $c = (i_1 \; \cdots \; i_k)$ be a cycle in $\pi$. Then, there exists a generalized grid polygon with boundary $(b'_1, \ldots, b'_m)$ and billiards permutation $\pi|_{[n] \setminus \{i_1, \ldots, i_k\}}$, such that when viewed as vectors, $b'_1, \ldots, b'_m$ are the same as $b_j$ for each $j \notin \{i_1, \ldots, i_k\}$ in index order. We denote this generalized grid polygon by $P - c$. 
\end{theorem}
The reason for the formalism introduced in this section is to prove this result. In order to show that we can drop a cycle, it suffices to describe how the process modifies the tree of horizontal strips of the generalized grid polygon $P$. In particular, we will modify each horizontal strip of $P$ independently and then glue them back together. We need to show that the resulting tree of horizontal strips is valid and it corresponds to the generalized grid polygon $P - c$ by \Cref{lem_tree_reconstruction}.

\begin{proof}
    Since any cycle is contained in a component, removing it does not affect other components. Thus, it suffices to prove the case that $P$ is an indecomposable generalized grid polygon.
    
    By \Cref{lem_tree_of_strips} and \Cref{lem_tree_reconstruction}, it suffices to describe how the horizontal strips of $P$ are modified under the process of dropping $c$ and show that when assembled back together, they have the correct boundary and cycle structure. Let $\Gamma_{S}$ be the tree of horizontal strips of $P$. 

    Let $h$ be a horizontal strip whose top side has length $p$ and whose bottom side has length $q$. We modify $h$ as follows: shorten each of the four sides by half the number of light beams in $c$ that are incident to it (see \Cref{fig_shorten_strip}). More formally, we remove the $2$-simplices of $h$ crossed by a beam of $c$ going in the $60^{\circ}$ direction or the $180^{\circ}$ direction (see \Cref{fig_shorten_strip_2}). 
    
    Note that the resulting object, $h'$, may no longer be a horizontal strip: if the horizontal light beam between the two non-horizontal sides of $h$ is a beam of $c$, then $h'$ will have height $0$. We call such an object a \defterm{degenerate horizontal strip}. We must also check that the resulting lengths of the four sides indeed form a (possibly degenerate) horizontal strip. 

    \begin{figure}[ht]
        \centering
        \includegraphics[width=0.6 \textwidth]{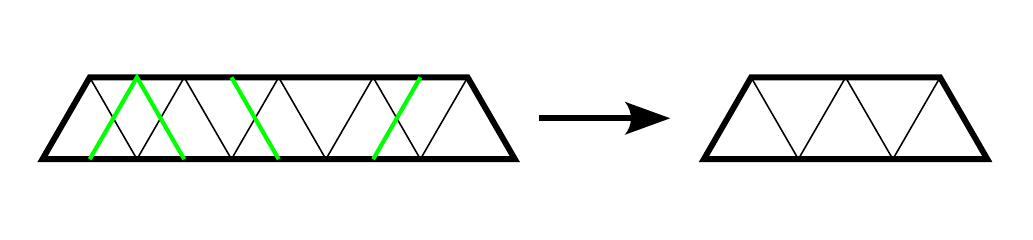}
        \caption{An example of how a horizontal strip is modified when a cycle $c$ is dropped. Here, both the top side and the bottom side are hit by $c$ four times, so they are shortened by two.}
        \label{fig_shorten_strip}
    \end{figure}
    
    \begin{enumerate}
        \item Suppose that the horizontal beam in $h$ is a beam of $c$. Then the two non-horizontal sides will have length zero, and we need to check that the top and bottom sides will have the same length. Indeed, if $p = q + 1$, then the horizontal beam will bounce up to hit the top side twice. The remaining beams hit the top and bottom sides exactly once. The case $p = q - 1$ follows the same argument. If $p = q$, then the horizontal beam will bounce to hit the top side and the bottom side once each. 
        \item Suppose that the horizontal beam is not in $c$. Then each beam hits the top side and the bottom side exactly once, which means that $p$ and $q$ will decrease by the same amount. 
    \end{enumerate}

    \begin{figure}[ht]
        \centering
        \includegraphics[width=0.6 \textwidth]{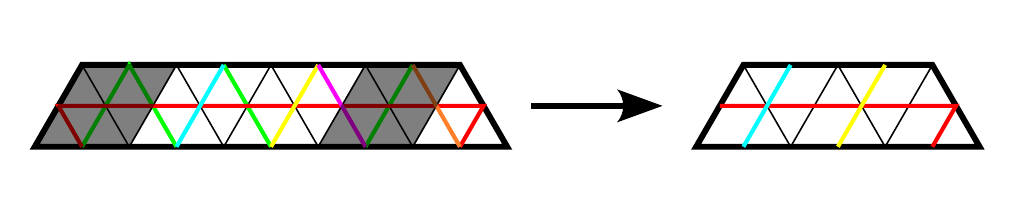}
        \caption{Another illustration of the modification process of $h$. The $2$-simplices crossed by beams of $c$ going in the $60^{\circ}$ direction or the $180^{\circ}$ direction are removed. Furthermore, the beams not in $c$ which go in the $60^{\circ}$ and $180^{\circ}$ directions are preserved.}
        \label{fig_shorten_strip_2}
    \end{figure}

    \begin{figure}[ht]
        \centering
        \includegraphics[width=0.7 \textwidth]{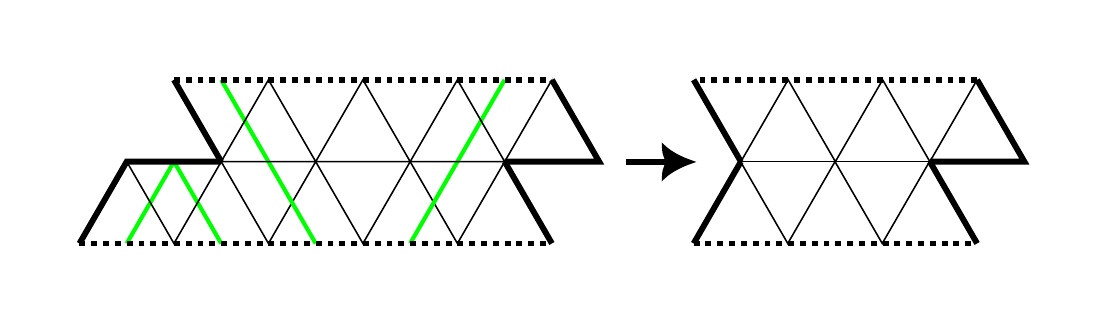}
        \caption{How two adjacent horizontal strips are modified.}
        \label{fig_two_strips_glued}
    \end{figure}

    Now we have a new set of (possibly degenerate) horizontal strips $S'$. We form $\Gamma_{S'}$ as follows. For an edge between $h_1, h_2$ labeled $\{p_1, p_2\}$ in $\Gamma_{S}$, add an edge between the modified strips $h'_1, h'_2$ labeled $\{p'_1, p'_2\}$ in $\Gamma_{S'}$, where $p'_1$ consists of the $1$-simplices of $p_1$ not hit by a beam of $c$ going in the $60^{\circ}$ direction or the $180^{\circ}$ direction (see \Cref{fig_two_strips_glued}). If $p'_1$ now has length zero (i.e., is a vertex), we add the orientation $\alpha$ to the edge label. 

    A complication is that our current tree of horizontal strips may contain degenerate horizontal strips, which must be removed while maintaining a valid tree structure.

    \begin{figure}[ht]
        \centering
        \includegraphics[width=0.8 \textwidth]{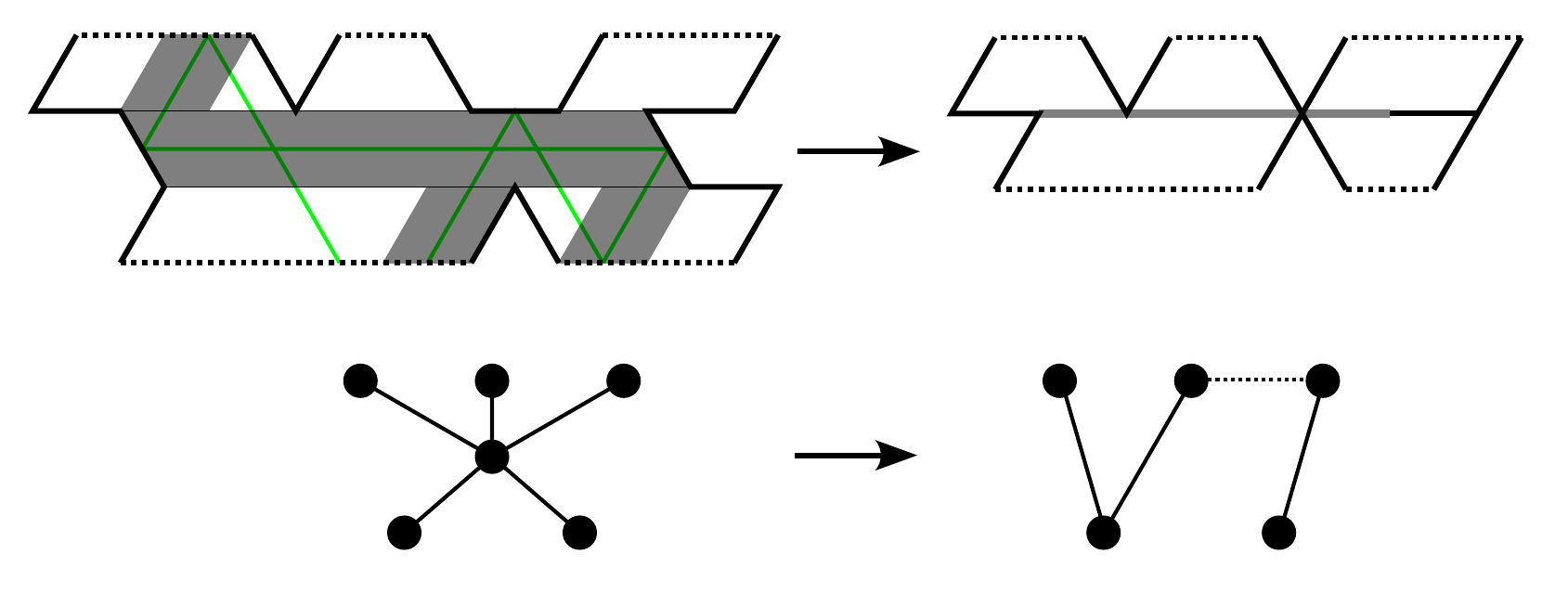}
        \caption{When we remove a degenerate strip from $\Gamma_{S'}$, we modify the neighborhood so that the graph remains a tree. The modified strips must remain connected and contractible, so it is possible to add edges labeled by a point and an orientation to make the graph connected (we added the dotted edge in the diagram).}
        \label{fig_delete_degenerate}
    \end{figure}
    
    Suppose $h$ is a horizontal strip of $P$ which becomes a degenerate horizontal strip $h'$. We remove $h'$ from $\Gamma_{S'}$ and modify the neighborhood of $h'$ to form a new graph $\Gamma_{S' \setminus \{h'\}}$. Let $N$ be the neighborhood of $h$ in $\Gamma_{S}$ and $N'$ the set of modified strips. Since $N$ is an independent set in $\Gamma_{S}$, we would like to show that $N'$ forms a tree in $\Gamma_{S' \setminus \{h'\}}$. We first add edges between horizontal strips labeled by maximal paths which collapse to the same path when $h$ is collapsed. Observe that collapsing $h$ does not change the contractibility of our simplicial complex, so this induced subgraph cannot have a cycle. Furthermore, for any horizontal panes $b$ and $b'$ of $h$ that are collapsed together, they cannot both be on the boundary of $P$ since that would create a $2$-cycle. This means that $N'$ is connected, so we can add edges labeled by single points and the appropriate orientation to make the induced subgraph on $N'$ connected (see \Cref{fig_delete_degenerate}).

    Let $Q$ denote the generalized grid polygon obtained from \Cref{lem_tree_reconstruction} using the modified tree structure above. We now show that $Q$ has boundary $(b'_1, \ldots, b'_m)$ and billiards permutation $\pi|_{[n] \setminus \{i_1, \ldots, i_k\}}$. The boundary of $Q$ is as desired because for any pane on the boundary of $P$, it is deleted if and only if it is a part of $c$, because the number of beams of $c$ hitting a boundary pane is either $2$ or $0$. It follows that the boundary of $Q$ is traced by the vectors corresponding to the boundary panes of $P$ not hit by $c$ in counterclockwise order. 

    \begin{figure}[ht]
        \centering
        \includegraphics[width=0.65 \textwidth]{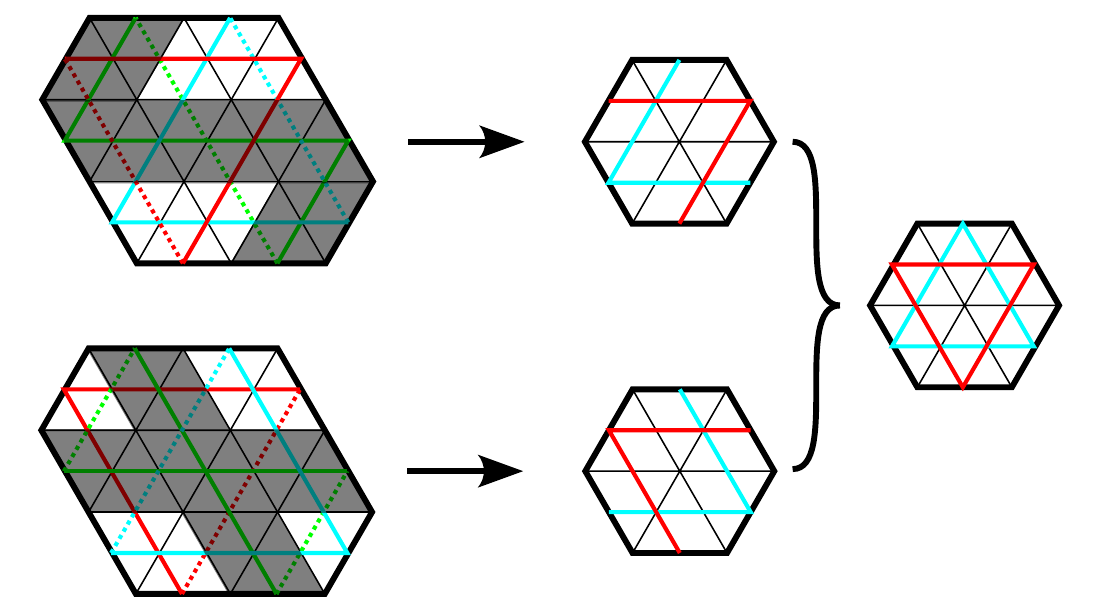}
        \caption{Two different ways of viewing the modification of horizontal strips recover the billiards permutation.}
        \label{fig_two_ways}
    \end{figure}
    
    To prove that the billiards permutation also comes from the restriction of $\pi$, let $b$ and $b'$ be two boundary panes of $P$ not hit by $c$. Suppose that $b$ is mapped to $b'$ by $\pi$. If the beam from $b$ to $b'$ is horizontal, then they are in the same horizontal strip, and remain in the same horizontal strip so that $b$ still maps to $b'$. 
    
    Now suppose the beam goes in the $60^{\circ}$ direction. Recall that we can view the modification of horizontal strips as deleting the $2$-simplices hit by $60^{\circ}$ and $180^{\circ}$ beams of $c$, which does not affect the $60^{\circ}$ beams not in $c$ (see \Cref{fig_shorten_strip_2}). Equivalently, we can view the modification of horizontal strips as deleting the $2$-simplices hit by $240^{\circ}$ and $180^{\circ}$ beams of $c$, which shows that the $240^{\circ}$ beams not in $c$ are also preserved (see \Cref{fig_two_ways}). This completes the proof. 
\end{proof}

\section{The perimeter of a grid polygon}\label{sec_perim}
In this section, we prove \Cref{thm_perim}. We do so by proving the following stronger theorem involving generalized grid polygons (\Cref{def_polygon}).
\begin{theorem}\label{thm_perim_strong}
    Let $P$ be a generalized grid polygon. Then
    \[
        \cyc(P) \leq \frac{1}{4}(\perim(P) + 2\comps(P)),
    \]
    where $\comps(P)$ is the number of components of $P$.
    Furthermore, if $P$ is indecomposable primitive, then 
    \[
        \cyc(P) = \frac{1}{4}(\perim(P) + 2)
    \]
    if and only if the billiards permutation of $P$ contains exactly two $3$-cycles and no cycles of length greater than $4$. 
\end{theorem}

It suffices to prove the above for indecomposable generalized grid polygons. The proof idea proceeds as follows: we remove an $m$-cycle, where $m \geq 4$, and apply the induction hypothesis to the resulting generalized grid polygon (given by \Cref{thm_dropping_cycle}) by inductively examining an outermost component. We would like to control the number of components of the new generalized grid polygon. Thus, we want to understand how removing a cycle can create more components. 

\begin{lemma}\label{lem_cutting_path}
    Let $P$ be an indecomposable generalized grid polygon with boundary $(b_1, \ldots, b_n)$, and let $c_1, \ldots, c_a$ be some cycles in its billiards permutation. Let $Q = P - c_1 - \cdots - c_a$ be the generalized grid polygon obtained by removing $c_1, \ldots, c_a$ from $P$ (as given by \Cref{thm_dropping_cycle}) and suppose that $Q$ is decomposable. Let $Q_1$ be an outermost component of $Q$. Suppose that $Q_1$ has boundary $(b'_1, \ldots, b'_\ell)$, where $b'_1$ corresponds to $b_1$ and $b'_{\ell}$ corresponds to $b_k$. Then there exists an interior path from the tail of $b_1$ to the head of $b_k$ in $P$ such that its image under $f$ is a shortest grid path and each pane in the path is crossed only by beams in the cycles $c_1, \ldots, c_a$ (see \Cref{fig_components}). 

    \begin{figure}[ht]
        \centering
        \includegraphics[width=0.5 \textwidth]{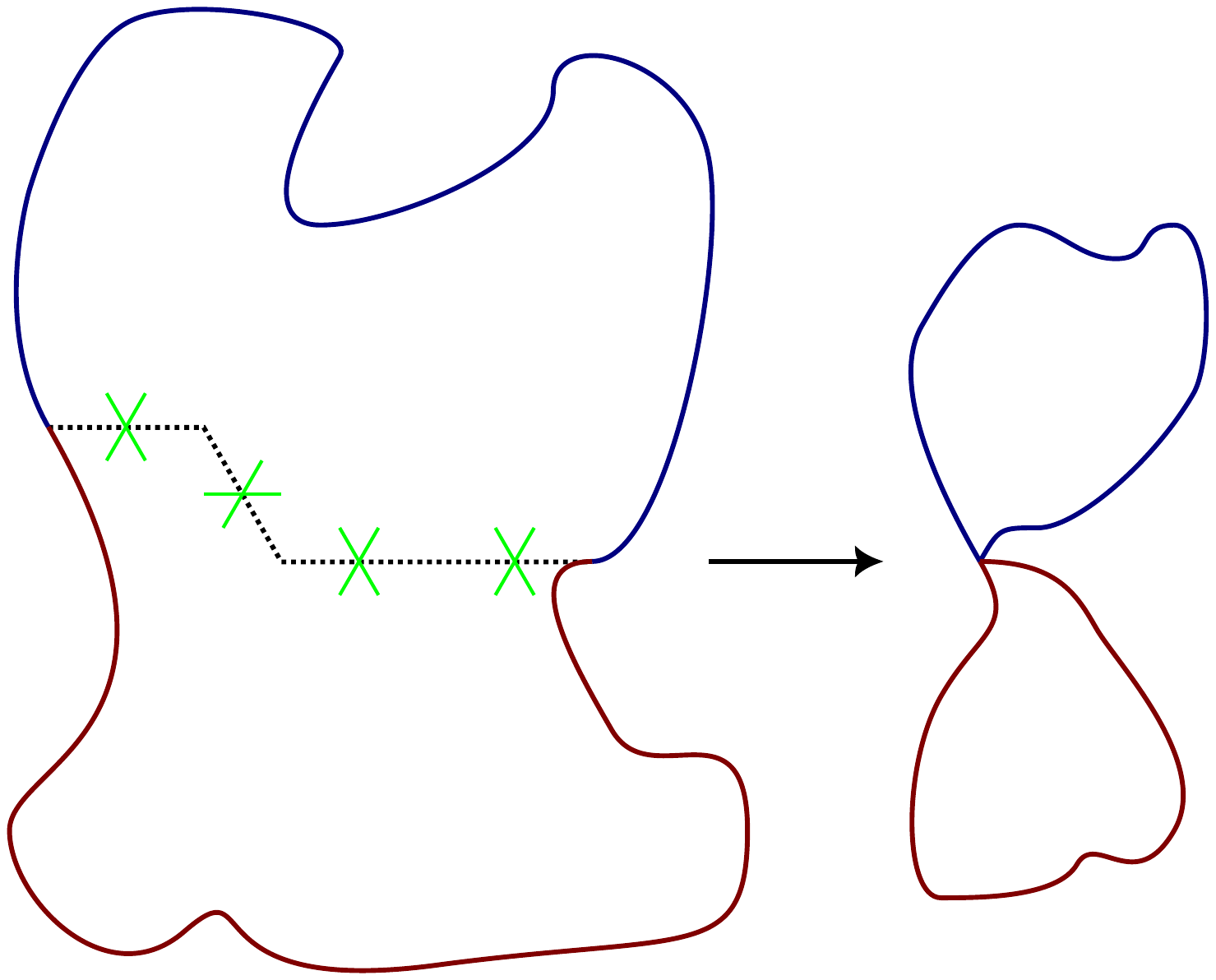}
        \caption{Illustration of an interior shortest path dividing a polygon into two parts which become separate components after removing the green cycle.}
        \label{fig_components}
    \end{figure}
\end{lemma}
\begin{proof}
    Let $x$ be the tail of $b_1$ and $y$ the head of $b_k$. We claim that no beam from a cycle with all its panes in $\{b_1, \ldots, b_k\}$ will intersect a beam from a cycle with all its panes in $\{b_{k + 1}, \ldots, b_n\}$. This is because they will still intersect after $c_1, \ldots, c_a$ are removed and violate the assumption that $Q_1$ is a component of $Q$. Also, no cycle other than $c_1, \ldots, c_a$ can contain both a pane from $\{b_1, \ldots, b_k\}$ and one from $\{b_{k + 1}, \ldots, b_n\}$. Thus, starting from $x$, it is possible to find an interior path (of $P$) to $y$ that does not intersect any beam other than those of $c_1, \ldots, c_a$. 

    \begin{figure}[ht]
        \centering
        \includegraphics[width=0.65 \textwidth]{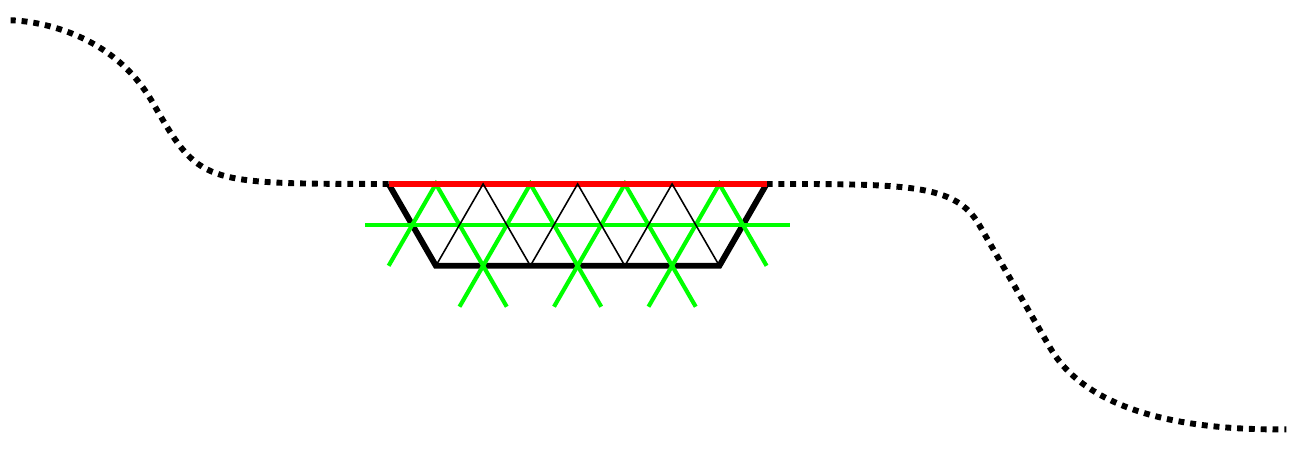}
        \caption{Whenever there is a subpath in the path that looks like the black segments above, we can replace it with a shorter straight segment like the red segment above
        and still have a path that only intersects the cycles $c_1, \ldots, c_a$.}
        \label{fig_shorten_path}
    \end{figure}

    If this path is not a shortest path, we can shorten it by replacing a subpath of the form $(u, v_1, \ldots, v_\ell, w)$, where they form three sides of a trapezoid and $v_1, \ldots, v_{\ell}$ all go in the same direction, by the fourth side, which is strictly shorter (see \Cref{fig_shorten_path}). The new path still does not intersect any beam other than those of $c_1, \ldots, c_a$. If the new path no longer stays in the interior and contains some boundary pane $b$, then the path divides the polygon into at least three components, each of which becomes a component of $Q$ when panes incident to $c_1, \ldots, c_a$ are removed. This violates the assumption that $Q_1$ is an outermost component. Thus, we can eventually arrive at a shortest interior path, as desired. 
\end{proof}
Using \Cref{lem_cutting_path}, we can inductively cut an indecomposable generalized grid polygon $P$ into pieces which will turn into the components of $Q$ when $c_1, \ldots, c_a$ are removed (see \Cref{fig_pre_comps}), by cutting out the piece corresponding to an outermost component of the remaining part of $Q$ at each step. Let $P_1, \ldots, P_k$ be these pieces, in the order they are obtained. We call these a sequence of \defterm{pre-components} of $Q$. Observe that depending on the way we cut $P$, this sequence of pre-components may not be unique. On the other hand, $P_1$ always corresponds to an outermost component of $Q$. Furthermore, for each $i$, $P_{i + 1}, \ldots, P_k$ is a sequence of pre-components of the generalized grid polygon consisting of the components of $Q$ corresponding to $P_{i + 1}, \ldots, P_k$. In future discussions, we often fix one such sequence and refer to them as ``the'' pre-components of $Q$.

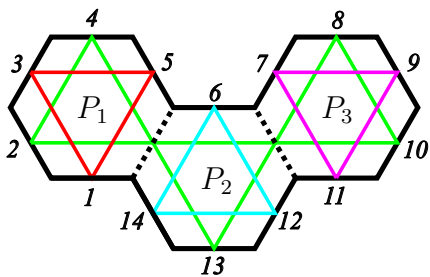
\begin{figure}[ht]
\centering
\begin{tikzpicture}
	\node (image) at (0,0) {
		\includegraphics[width=0.4\textwidth]{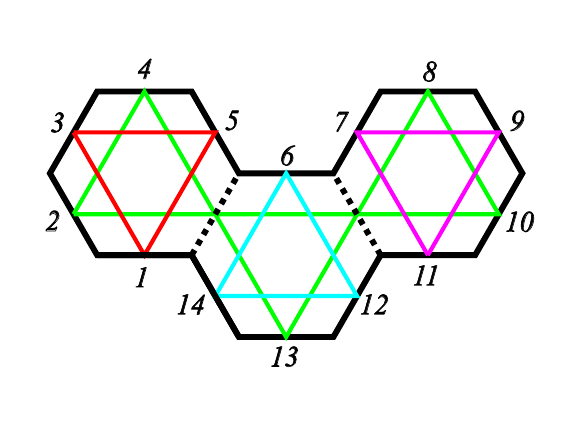}
	};
    \node[text width=3cm] at (-0.35,0.4) 
    {$P_1$};
    \node[text width=3cm] at (1.3,-0.55) 
    {$P_2$};
    \node[text width=3cm] at (2.9,0.4) 
    {$P_3$};
\end{tikzpicture}
\caption{This is one way pre-components can be assigned for the generalized grid polygon resulted from removing the green cycle.}
\label{fig_pre_comps}
\end{figure}

It follows from \Cref{lem_cutting_path} that if a pane is on the boundary of a pre-component but not on the boundary of $P$, it is crossed only by the cycles $c_1, \ldots, c_a$. Note that for a pre-component $P_i$, when viewed as a generalized grid polygon, the beams of the cycle $c_j$ inside $P_i$ may break up into several cycles. This is why we proved \Cref{lem_cutting_path} for removing several cycles. 

When we prove \Cref{thm_perim_strong}, we will induct on $\perim(P)$: we drop an $m$-cycle for some $m \geq 4$ and apply induction hypothesis inside these pre-components. To do that, we need to understand what cycles look like inside a pre-component.
\begin{lemma}\label{lem_k+2}
    Let $P$ be an indecomposable generalized grid polygon and $Q = P - c_1 - \cdots - c_a$ the generalized grid polygon obtained by removing $c_1, \ldots, c_a$ from $P$. Let $P_1, \ldots, P_k$ be a sequence of pre-components of $Q$. Suppose that the path on the boundary of $P_1$ but not on the boundary of $P$ is a shortest grid path of length $k \geq 2$ (by \Cref{lem_cutting_path}). Then the number of boundary panes of $P$ inside $P_1$ hit by $c_1, \ldots, c_a$ is at least $k + 2$.
\end{lemma}
\begin{proof}
    Suppose that the path on the boundary of $P_1$ but not on the boundary of $P$ has panes $p_1, \ldots, p_k$ going from left to right. By \Cref{lem_cutting_path}, each $p_i$ is crossed by two beams of $c_1, \ldots, c_a$. 

    \begin{figure}[ht]
        \centering
        \includegraphics[width=0.4 \textwidth]{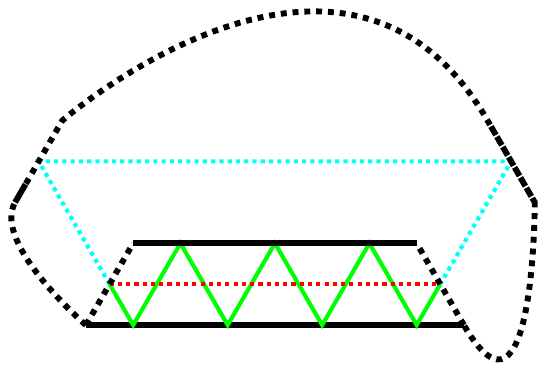}
        \caption{It is impossible to only have $k + 1$ boundary panes hit by $c_1, \ldots, c_a$. Either we end up with a horizontal strip (red), or our cycle encloses some boundary panes (blue).}
        \label{fig_k+2_straight}
    \end{figure}
    
    First suppose that this path is straight and horizontal. Let $u_1, \ldots, u_k$ be the boundary panes of $P_1$ hit by the northeast beams from $p_1, \ldots, p_k$, respectively. Similarly, let $v_1, \ldots, v_k$ be the boundary panes of $P_1$ hit by the northwest beams from $p_1, \ldots, p_k$, respectively. Clearly, $u_1, \ldots, u_k$ are all distinct and $v_1, \ldots, v_k$ are all distinct. Furthermore, $v_1$ must be distinct from each $u_i$. Suppose for the sake of contradiction that these $k + 1$ panes, $v_1, u_1, \ldots, u_k$, are the only boundary panes in $P_1$ hit by $c_1, \ldots, c_a$. Then we must have 
    \[
        v_2 = u_1, \, v_3 = u_2, \, \ldots, \, v_k = u_{k - 1}.
    \]
    Hence, the beam from $p_k$ to $u_k$ must bounce to $v_1$ and bounce back to $p_1$. However, this either results in $P_1$ being a horizontal strip, in which case it will become an empty component of $Q$, or the cycle will encircle some boundary panes, violating the simply-connectedness of $P$ (see \Cref{fig_k+2_straight}). Therefore, there must be at least $k + 2$ boundary panes hit by $c_1, \ldots, c_a$. 

    \begin{figure}[ht]
        \centering
        \includegraphics[width=0.35 \textwidth]{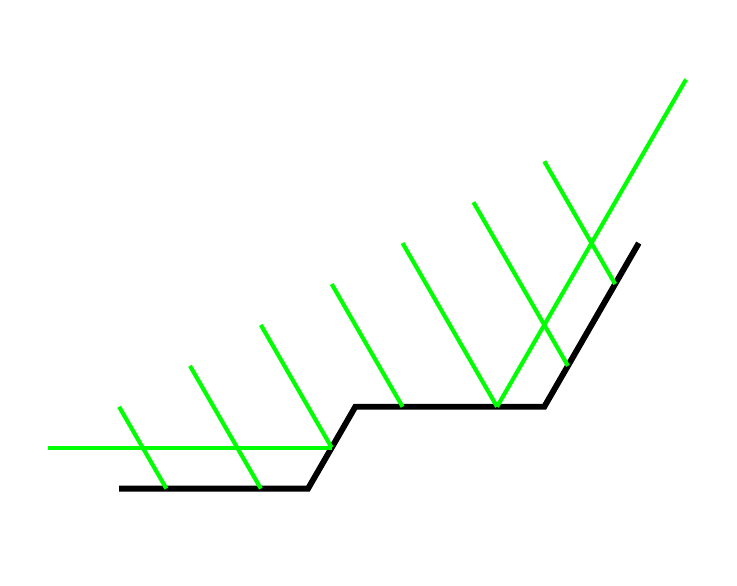}
        \caption{There must be at least $k + 2$ distinct panes on the boundary hit by $c_1, \ldots, c_a$.}
        \label{fig_k+2_bent}
    \end{figure}

    Now suppose that the path consists of panes going in the $180^{\circ}$ and $240^{\circ}$ direction. Then each pane has a beam going in the northwest direction. Let the boundary panes hit by these beams be $q_1, \ldots, q_k$. Let $w_1$ be the boundary pane hit by the beam going west from the first pane in the path that is in the $240^{\circ}$ direction. Let $w_2$ be the boundary pane hit by the beam going northeast from the last pane in the path in the $180^{\circ}$ direction. We claim that $q_1, \ldots, q_k, w_1, w_2$ are all distinct (see \Cref{fig_k+2_bent}). First, $q_1, \ldots, q_k$ are all distinct and $w_1 \neq w_2$. The only possibility for $w_1$ to be in $q_1, \ldots q_k$ is if $w_1 = q_1$. However, this would mean that $q_1$ is a pane adjacent to the path and hit by $c_1, \ldots, c_a$. This is impossible by \Cref{lem_cutting_path}. Thus, $w_1 \notin \{q_1, \ldots, q_k\}$. Similarly, $w_2 \notin \{q_1, \ldots, q_k\}$. 
\end{proof}

Now, we start by proving \Cref{thm_perim_strong} for the simplest polygons: those that only contain $3$-cycles in their billiards permutation. 
\begin{lemma}\label{lem_triangles}
    Let $P$ be an indecomposable generalized grid polygon. If the billiards permutation of $P$ contains only $3$-cycles, then $P$ is either a unit triangle or a unit hexagon consisting of six unit triangles (see \Cref{fig_only_three_cycles}). 

    \begin{figure}[ht]
        \centering
        \includegraphics[width=0.3 \textwidth]{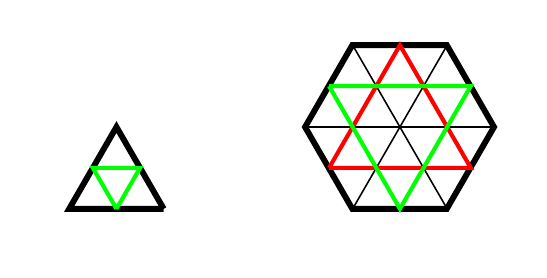}
        \caption{Any indecomposable generalized grid polygon containing only $3$-cycles maps to a rotation of one of these two in the triangular grid.}
        \label{fig_only_three_cycles}
    \end{figure}
\end{lemma}
\begin{proof}
    Since a $3$-cycle does not cross any interior pane twice, by \Cref{lem_cutting_path}, removing a $3$-cycle results in an indecomposable generalized grid polygon. Suppose for the sake of contradiction that $P$ contains more than two $3$-cycles. Then at least two of the $3$-cycles in $P$ have the same orientation. Applying \Cref{thm_dropping_cycle}, we can sequentially remove all the $3$-cycles but these two to get an indecomposable generalized grid polygon of perimeter $6$ with two $3$-cycles of the same orientation (suppose they are both downwards). Furthermore, the boundary must contain exactly two panes going in the $60^{\circ}$, $180^{\circ}$, and $300^{\circ}$ directions, respectively. There are 
    \[
        \binom{6}{2}\binom{4}{2}\binom{2}{2}/6 = 15
    \]
    possible boundary loops. It can be checked easily that none of them can admit a cycle structure with two $3$-cycles of the same orientation. 

    Therefore, any indecomposable generalized grid polygon containing only $3$-cycles contains at most two $3$-cycles, and if it contains exactly two, they must have opposite orientation. 
\end{proof}

We now bound the number of components that an indecomposable generalized grid polygon can have after removing several cycles. Suppose $Q$ is obtained by removing some cycles of total length $m$ (the cycles hit $m$ boundary panes in total) from the indecomposable generalized grid polygon $P$ and that $P_1, \ldots, P_k$ are the pre-components of $Q$. We call $P_i$ a \defterm{triangular pre-component} if the parts of the cycles to be removed inside $P_i$ form a single triangle. Recall that we say a generalized grid polygon is \defterm{primitive} if it can be disconnected by a single interior pane. 
\begin{proposition}\label{prop_bound_on_comps}
    Let $c_1, \ldots, c_a$ be cycles with total length $m$ in an indecomposable generalized grid polygon $P$ and let $Q = P - c_1 - \cdots - c_a$. Fix a sequence of pre-components of $Q$. Let $k$ be the number of pre-components and $t$ the number of triangular pre-components. Then $2k - t \leq m - 2$. Furthermore, if $P$ is primitive, then equality holds if and only if $m \leq 4$. 
\end{proposition}
\begin{proof}
    We induct on $m$. If $m = 3$, \Cref{lem_cutting_path} implies that $k = t = 1$. 
    
    Let $P_1, \ldots, P_k$ be the sequence of pre-components. Recall that $P_1$ corresponds to an outermost component of $Q$. If $P_1$ is a triangular pre-component, then letting $P'$ be $P$ with $P_1$ removed, we see that $c_1, \ldots, c_a$ may merge or break up into some other set of cycles which have total length $m - 1$. Applying the induction hypothesis, we have 
    \[
        2(k - 1) - (t - 1) \leq (m - 1) - 2,
    \]
    which is precisely what we want. 

    If $P_1$ is not a triangular pre-component, suppose that $P_1 \cap P'$ is a path of length $\ell$. If the total number of panes incident to $c_1, \ldots, c_a$ in $P_1$ and $P'$ are $m_1$ and $m_2$, respectively, then $m_1 + m_2 = m + 2 \ell$. By \Cref{lem_k+2}, $m_1 \geq 2 \ell + 2$ (if $\ell = 1$, $m_1 \geq 4$ by the assumption that $P_1$ is not triangular). Then $m_2 \leq m - 2$. Applying the induction on $P'$,
    \[
        2(k - 1) - t \leq m_2 - 2 \leq m - 4,
    \]
    which is again what we want. 

    Now if $P$ is primitive, any cut must have length at least $2$. If $k \geq 2$, when we induct as above to the last remaining pre-component, we are left with a pre-component that contains at least $6$ panes incident to $c_1, \ldots, c_a$. However, this means that equality cannot hold. Thus, we must have $k = 1$, and the only way equality holds is if $m \leq 4$. 
\end{proof}

Finally, we are in a position to prove \Cref{thm_perim_strong}.
\begin{proof}[Proof of \Cref{thm_perim_strong}]
    It suffices to prove this for $P$ indecomposable. We induct on the perimeter of $P$. Let $\pi$ be the billiards permutation of $P$. Recall that $\pi$ contains no $2$-cycles. If $\pi$ contains only $3$-cycles, \Cref{lem_triangles} implies the inequality. Now suppose $\pi$ contains some $m$-cycle $c$, where $m \geq 4$. If $\pi$ contains no $3$-cycles, the inequality immediately holds. Thus, we may assume $\perim(P) \geq 7$. 

    Let $Q = P - c$ be the generalized grid polygon obtained by removing $c$ from $P$. Let $P_1, \ldots, P_k$ be a sequence of pre-components of $Q$, with corresponding components $Q_1, \ldots, Q_k$, and let $t$ be the number of triangular pre-components. By the induction hypothesis, for each $i \in [k]$, we have
    \[
        \cyc(Q_i) \leq \frac{1}{4} (\perim(Q_i) + 2).
    \]
    Furthermore, if $P_i$ is a triangular pre-component, then 
    \begin{align*}
        \cyc(Q_i) 
        &= \cyc(P_i) - 1 \\
        &\leq \frac{1}{4} (\perim(P_i) + 2) - 1\\
        &= \frac{1}{4} (\perim(Q_i) + 1).
    \end{align*}
    Since $\perim(P) = \sum_{i = 1}^{k} \perim(Q_i) + m$ and $\cyc(P) = \sum_{i = 1}^{k} \cyc(Q_i) + 1$, it follows from \Cref{prop_bound_on_comps} that
    \begin{align*}
        \cyc(P) 
        &= \sum_{i = 1}^{k} \cyc(Q_i) + 1\\
        &\leq \sum_{i = 1}^{k} \frac{\perim(Q_i)}{4} + \frac{2k - t}{4} + 1\\
        &\leq \frac{1}{4} (\perim(P) + 2).
    \end{align*}
    If $m > 4$ and $P$ is primitive, \Cref{prop_bound_on_comps} implies that the above inequality is strict. Thus, a connected primitive equality case must contain only $3$-cycles and $4$-cycles. 
    \end{proof}

\section{Future Directions}\label{sec_future}
As remarked in \cite{defant2023} and \cite{adams2024}, the study of combinatorial billiards is a new and rather unexplored topic. There are many interesting avenues for future research in this area, such as considering billiards systems in different grids, studying polygonal regions with holes, further investigating connections with plabic graphs (see \cite{defant2023} for more discussions on these directions), and studying billiards systems on other combinatorial objects (see \cite{adams2024}). Below, we discuss some potential future directions more relevant to the contents of this article. 

By \Cref{thm_perim}, the primitive equality cases contain exactly two $3$-cycles and no cycles of length greater than $4$ in their billiards permutations. It is a natural question to fully classify the equality cases. 
\begin{question}\label{qst_equality_cases}
    Is there a way to characterize all primitive simple grid polygons (or generalized grid polygons) which have exactly two $3$-cycles and no cycles of length greater than $4$? 
\end{question}

If we look at the $1$-skeleton of a generalized grid polygon (see \Cref{def_polygon}), we obtain a connected undirected plane graph (a planar graph with a particular drawing on the plane) such that every face except the outer face is a triangle. This is because a wedge of disks along boundary points has an embedding in the plane. Lam and Postnikov \cite{lam2020} call a plane graph which only has triangular faces a \defterm{cactus}. They note that a cactus (perhaps more precisely the geometric realization of a generalized grid polygon) is either a single edge, a triangulated disk, or a wedge of smaller cacti along their boundary vertices. 

\begin{definition}[\cite{lam2020}]\label{membrane_A2}
    A \defterm{membrane of type $A_2$} with \defterm{boundary loop} $(b_1, \ldots, b_n)$ is a pair $(G, f)$ where $G$ is a cactus on a vertex set $V$ with a sequence of boundary vertices $x_1, \ldots, x_n$ and $f$ is a map $f: V \to T$ to the triangular grid such that 
    \begin{enumerate}
        \item the endpoints of any edge of $G$ are mapped by $f$ to endpoints of a pane, and
        \item for each $i \in [n]$, $f(x_{i})$ is the tail of $b_i$.
    \end{enumerate}
    A membrane of type $A_2$ is \defterm{minimal} if it has the minimum possible area of any membrane of type $A_2$ with the same boundary loop. 
\end{definition}

Every generalized grid polygon is a minimal membrane of type $A_2$ by the above. In fact, they can be viewed as simply-connected simplicial complexes (not necessarily homogeneous) satisfying \Cref{def_polygon} (1) and (2). \Cref{fig_good_eg_2} shows an example of a minimal membrane of type $A_2$ which is not a generalized grid polygon. We can similarly define billiards permutations on minimal membranes of type $A_2$. It is not hard to see that the billiards permutation is exactly Postnikov's \cite{postnikov2006} strand permutation on plabic graphs under the map taking membranes to plabic graphs given in \cite{lam2020}. We conjecture that \Cref{thm_perim_strong} holds for this larger class of objects. 

\begin{conjecture}\label{conj_membranes}
    Let $P$ be a minimal membrane of type $A_2$. Then
    \[
        \cyc(P) \leq \frac{1}{4}(\perim(P) + 2\comps(P)).
    \]
    Furthermore, if $P$ is indecomposable, primitive, and not a single edge, then 
    \[
        \cyc(P) = \frac{1}{4}(\perim(P) + 2)
    \]
    if and only if the billiards permutation of $P$ contains exactly two $3$-cycles and no cycles of length greater than $4$. 
\end{conjecture}
The reason that our proof does not work for all minimal membranes of type $A_2$ is that the horizontal strips of a minimal membrane of type $A_2$ may not have the tree structure described in \Cref{lem_tree_of_strips} (for example, the horizontal strips of \Cref{fig_good_eg_2} form a cycle of length four).

Defant and Jiradilok's motivation for their paper \cite{defant2023} was to understand strand permutations of plabic graphs. Our main result suggests that studying the dual objects---billiards permutations and membranes---could prove productive due to the geometric properties that they have. Besides \Cref{thm_perim_strong}, another illustration of the usefulness of generalized grid polygons in the study of plabic graphs is a simpler proof of the following result of Defant and Jiradilok \cite{defant2023}. 
\begin{theorem}[\cite{defant2023}]\label{thm_area}
    If $P$ is a simple grid polygon, then 
    \[
        \cyc(P) \leq \frac{1}{6}\area(P) + 1,
    \]
    where $\area(P)$ is the number of unit triangles contained in $P$ (this is the Euclidean area of $P$ multiplied by the normalizing factor of $4/\sqrt{3}$). Furthermore, equality holds if and only if $P$ is a tree of unit hexagons. 
\end{theorem}
\begin{proof}
    As in \Cref{thm_perim}, we prove this result more generally for generalized grid polygons:
    \[
        \cyc(P) \leq \frac{1}{6}\area(P) + \comps(P),
    \]
    where $\comps(P)$ is the number of components of $P$. 
    
    The following generalization of a lemma of Defant and Jiradilok (\cite[Lemma~3.1]{defant2023}) is straightforward and we skip its proof. Let $P$ be a generalized grid polygon which is the union of two generalized grid polygons $P_1, P_2$ which intersect at a single $1$-simplex. Then $\cyc(P) = \cyc(P_1) + \cyc(P_2) - 1$. This result reduces the inequality to the case where $P$ is indecomposable and primitive. 

    We use \Cref{thm_dropping_cycle} to prove the theorem via induction on $\area(P)$. It is easy to check that a unit hexagon satisfies the equality and any indecomposable primitive generalized grid polygon with only once cycle satisfies the inequality strictly. Suppose $P$ has at least two cycles and is not a unit hexagon. Suppose we remove some $m$-cycle $c$ and $P-c$ has components $Q_1, \ldots, Q_k$. Denote by $\area(c)$ the number of $2$-simplices removed in the process of dropping $c$. Then
    \begin{align*}
        \cyc(P) 
        &= 1 + \sum_{i = 1}^{k} \cyc(Q_i) \\
        &\leq 1 + \sum_{i = 1}^{k} \left(\frac{1}{6}\area(Q_i) + 1 \right) \\
        &= k + 1 + \frac{1}{6}(\area(P) - \area(c)).
    \end{align*}
    Therefore, it suffices to prove that any cycle $c$ has  $\area(c) > 6k$. Observe that since $P$ is indecomposable and primitive, if $k = 1$, then the only cycle with $\area(c) \leq 6$ is the $3$-cycle in a unit hexagon. But the only indecomposable primitive generalized grid polygon with such a $3$-cycle is a unit hexagon. (The second smallest cycle has $\area(c) = 9$). 

    Now suppose $k \geq 2$. Let $P_1, \ldots, P_k$ be a sequence of pre-components which correspond to $Q_1, \ldots, Q_k$. Since $P$ is primitive, all the shortest paths from \Cref{lem_cutting_path} separating these pre-components have length at least $2$, and there is no triangular pre-component. Thus, by \Cref{prop_bound_on_comps}, each pre-component satisfies \Cref{lem_k+2}, with at least one that is not equality case. We can check that when the path has length at least two, at least six $2$-simplices are deleted in the process of dropping the cycle $c$, which completes the proof. 
\end{proof}

\begin{figure}[ht]
    \centering
    \includegraphics[width=0.3 \textwidth]{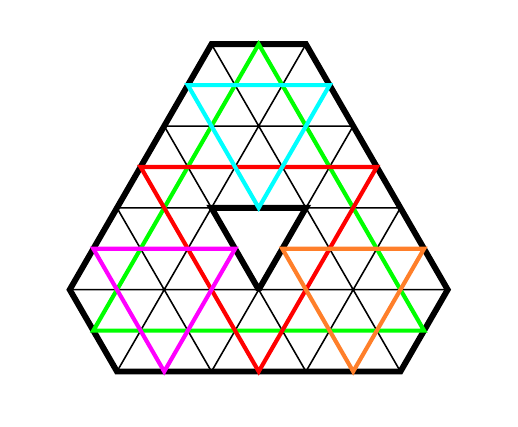}
    \caption{A grid polygon with one hole. The billiards permutation of this polygon has five $3$-cycles.}
    \label{fig_holeygon}
\end{figure}

A natural question to consider is the generalization of \Cref{thm_perim} and \Cref{thm_area} for grid polygons with holes. It is not hard to extend \Cref{def_polygon} to account for these shapes and define billiards permutations on them. As noted by Defant and Jiradilok \cite{defant2023}, the there exists grid polygons with holes violating \Cref{thm_perim} and \Cref{thm_area} as written (see \Cref{fig_holeygon}). We have the following conjectural analogues of \Cref{thm_perim} and \Cref{thm_area} for grid polygons with holes. 

\begin{conjecture}\label{conj_perim}
    If $P$ is a grid polygon with holes, then
    \[
        \cyc(P) \leq \frac{1}{4}(\outerperim(P) + 2 \innerperim(P) + 2),
    \]
    where $\outerperim(P)$ is the perimeter of the outer boundary of $P$ and $\innerperim(P)$ is the total perimeter of the holes.  
\end{conjecture}
Note that \Cref{fig_holeygon} is an equality case of \Cref{conj_perim}. 
\begin{conjecture}\label{conj_area}
    If $P$ is a grid polygon with holes, then
    \[
        \cyc(P) \leq \frac{1}{6}\area(P) + \holes(P) + 1,
    \]
    where $\holes(P)$ is the number of holes in $P$. 
\end{conjecture}

Another possible direction is considering analogues of \Cref{thm_perim} and \Cref{thm_area} for minimal membranes of type $A_3$ (see \cite{lam2020}) and higher. A membrane of type $A_3$ is a cactus graph equipped with a map to the $A_3$ root lattice (also called the face-centered cubic lattice). For example, the cactus graph in \Cref{fig_bad_eg_2} can now become a membrane of type $A_3$, mapping to three sides of a regular tetrahedron. We can define billiards permutations on membranes of type $A_3$ from strand permutations on the dual plabic graphs. However, since membranes of type $A_3$ are no longer mapped to a plane, the interpretation of the billiards permutation as light beams reflecting off of mirrors needs to be modified. It would be interesting if one can find a natural geometric interpretation for these billiards permutations. 

Finally, we can study properties of the set of billiards permutation more directly. For example, the permutation $(1 \; 3) (2 \; 4) \in S_4$ is not the billiards permutation of any generalized grid polygon. Doing so might allow us to understand or characterize the difference between the two sides of the inequality in \Cref{thm_perim_strong}.
\begin{question}
    Characterize the subset of permutations $S \subset S_n$ which appear as the billiards permutation of some generalized grid polygon of perimeter $n$. 
\end{question}

\bibliographystyle{amsplain0}
\bibliography{bib}

\end{document}